\documentclass[12pt,reqno]{amsart}

\usepackage{amssymb,latexsym}

\usepackage{enumerate}

 \usepackage[french,english]{babel}
 \usepackage{amsmath}
\usepackage{graphicx}
\usepackage{amssymb}
\usepackage{bbm}
\usepackage{amsthm,mathtools}
\usepackage{ulem}
\usepackage{geometry}
\usepackage{tikz-cd}
\usepackage{mathrsfs}
\usepackage[colorinlistoftodos]{todonotes}
\usepackage{enumitem}
\usepackage{verbatim}
\usepackage[foot]{amsaddr}

\makeatletter

\@namedef{subjclassname@2010}{

  \textup{2020} Mathematics Subject Classification}

\makeatother
\newtheorem{thm}{Theorem}[section]
\newtheorem*{thm*}{Theorem}

\newtheorem{cor}{Corollary}

\newtheorem{lem}[thm]{Lemma}

\theoremstyle{definition}

\numberwithin{equation}{section}

\newcommand{\mbr}{\mathbb{R}}

\newcommand{\mcs}{\mathcal{S}}

\newcommand{\mme}{\mathrm{e}}
\newcommand{\mmi}{\mathrm{i}}

\newcommand{\M}{{\mathcal M}}

\usepackage{hyperref}
\hypersetup{hypertex=true,colorlinks=true,linkcolor=blue,anchorcolor=blue,citecolor=blue}
\DeclareMathOperator\dif{d\!}

\newcommand{\newabstract}[1]{%
  \par\bigskip
  \csname otherlanguage*\endcsname{#1}%
  \csname captions#1\endcsname
  \item[\hskip\labelsep\scshape\abstractname.]
}

\begin{document}

\baselineskip=17pt

\title[Large zeta sums]{Large zeta sums}

\author{Zikang Dong}
\author{Weijia Wang}
\author{Hao Zhang}
\address[Zikang Dong]{School of Mathematical Sciences, Tongji University, Shanghai 200092, P. R. China}
\address[Weijia Wang]{Yanqi Lake Beijing Institute of Mathematical Sciences and Applications $\&$ Yau Mathematical Sciences Center, Tsinghua University, Beijing 101408, P. R. China}
\address[Hao Zhang]{School of Mathematics, Hunan University, Changsha 410082, P. R. China}
\email{zikangdong@gmail.com}
\email{weijiawang@tsinghua.edu.cn}
\email{zhanghaomath@hnu.edu.cn}

\date{}

\begin{abstract} 
In this article, we investigate the behaviour of values of zeta sums $\sum_{n\le x}n^{\mmi t}$ when $x$ and $t$ are large. We show some asymptotic behaviour and Omega results of  zeta sums, which are analogous to previous results of large character sums $\sum_{n\le x}\chi(n)$.
\end{abstract}


\maketitle

\section{Introduction}
Let $q$ be a large integer and $\chi\pmod q$ be any non-principal Dirichlet character. The study of character sums $\sum_{n\le x}\chi(n)$ has a long history. In 1918, P\'olya and Vinogradov proved independently the nontrivial upper bound 
$$\sum_{n\le x}\chi(n)\ll \sqrt q\log q.$$
This uniform upper bound remains the best possible up to the implied constant till now and is called the  P\'olya-Vinogradov inequality. Assume the generalized Riemann hypothesis is true, in 1977 Montgomery and Vaughan \cite{MV77} showed that
$$\sum_{n\le x}\chi(n)\ll \sqrt q\log_2 q.$$
This is the best possible conditional upper bound up to the implied constant. Also conditionally, Granville and Soundararajan \cite{GS01} showed that $\log x/\log_2q\to\infty$ as $q\to\infty$ implies $\sum_{n\le x}\chi(n)=o(x)$. 

Like character sums, zeta sums $\sum_{n\le x}n^{\mmi t}$ have many similar properties. It is not hard to show the  P\'olya-Vinogradov type inequality
$$\sum_{n\le x}n^{\mmi t}\ll\sqrt t\log t,$$
assuming $x<t$ are both large. Unconditionally, the Vinogradov-Korobov method yields that $\log x/(\log t)^{2/3}\to\infty$ as $t\to\infty$ implies $\sum_{n\le x}n^{\mmi t}=o(x)$.  Though the analogous results are less celebrated than those of character sums, the study of zeta sums is important as well. The bounds of zeta sums on a wide range of $x$ is related to those of the Riemann zeta function close to the 1-line, while obtaining a larger saving on a more limited range of $x$ is related to the values of the Riemann zeta function on the critical line.

These two kinds of sums can be modelled by the same sums of random multiplicative functions $\sum_{n\le x}X_n$, where $X_n$ is the Steinhaus random multiplicative functions.
Recently, based on his celebrated work on moments of random multiplicative functions, Harper \cite{Harper} unconditionally showed  the low moments of zeta sums (and also character sums) have ``better than squareroot cancellation": for $1\le x\le T$ and $0\le k\le 1$, 
$$\frac{1}{T}\int_0^T\Big|\sum_{n\le x}n^{\mmi t}\Big|^{2k}\dif t\ll\bigg(\frac{x}{1+(1-k)\sqrt{\log_2(10L_T)}}\bigg)^k,$$
where $L_T=\min\{x,T/x\}.$ This draws much new attention to zeta sums. The best known lower bounds are due to La Bret\`eche, Munsch and Tenenbaum \cite{DMT}.
Yang \cite{Yang} showed the following conditional asymptotic formula, which is an analogue of a similar result of character sums by Granville and Soundararajan \cite{GS01}. In 2019, Lamzouri \cite{La2019} also generalized their work to the sums of Hecke eigenvalues of holomorphic cusp forms.  Denote by ${\mathcal S}(y)$ the $y$-friable integers and define
$$\Psi(x,y):=\sum_{n\le x\atop n\in{\mathcal S}(y)}1,\;\;\;\;\Psi(x,y;t):=\sum_{n\le x\atop n\in{\mathcal S}(y)}n^{{\rm i}t}.$$
\begin{thm*}\cite[Theorem 5]{Yang} 
Assume the Riemann hypothesis and let $T$ be large. If $2\le x\le T$, $T+y+3\le t\le T^{1000}$ and $y\ge(\log T)^2(\log x)^2(\log_2T)^{12}$, then
$$\sum_{n\le x}n^{-\mmi t}=\Psi(x,y;t)+O\bigg(\frac{\Psi(x,y)}{(\log_2T)^2}\bigg),$$
and 
$$\Big|\sum_{n\le x}n^{-\mmi t}\Big|\ll\Psi(x,(\log T)^2(\log_2T)^{20}),\;\;\;\;\forall t\in[T+(\log T)^2(\log_2T)^{15},T^{1000}].$$
So if $\log x/\log_2T\to\infty$ as $T\to\infty$, then we have
$$\sum_{n\le x}n^{-\mmi t}=o(x),\;\;\;\;\forall t\in[T+(\log T)^2(\log_2T)^{15},T^{1000}].$$
\end{thm*}
Yang conjectured (see \cite[Conjecture 1]{Yang}), the conditions of the above Theorem can be extended to $t\asymp T$ and $y=(\log T+(\log x)^2)(\log_2T)^A$ for some positive $A$. This conjecture is very strong. Assuming Yang's conjecture, one can deduce very sharp upper bounds for the derivatives of the Riemann zeta function (See \cite[Theorem 7]{Yang}). Thus the asymptotic behaviour of values of zeta sums is very important. 

Without the Riemann hypothesis, we can show that, except a small set of $t$, zeta sums can be approximated by the sums over friable numbers.
\begin{thm}\label{thm1.1}
For all $1\le t\le T$ but a set of measure at most $T^{1-1/\log x}$, whenever $2\leq x\leq T^{\frac{1}{3}}$, $y\geq \log x\log T (\log_2 T)^5$ we have
$$\sum_{n\le x}n^{{\rm i}t}=\Psi(x,y;t)+O\bigg(\frac{\Psi(x,y)}{(\log_2 T)^2}\bigg)$$
For all $1\le t\le T$ but a set of cardinal at most $T^{1-1/(\log_2 x)^2}$, we have
$$\Big|\sum_{n\le x}n^{{\rm i}t}\Big|\le\Psi(x,(\log T+(\log x)^2)(\log_2T)^5) .$$ 
\end{thm}

Yang's conditional result is sharp, since we can show that $x=(\log T)^A$ for some $A>0$ implies $\max_{t\in[1,T]}|\sum_{n\le x}n^{{\rm i}t}|\gg x$. In fact, zeta sums can gain large values in any direction. Denote by $\rho(\cdot)$ the Dickman function.
\begin{thm}\label{thm1.2}
Suppose $\log x\le(\log_2T)^2/(\log_3T)^2$. For all $|\theta|\le\pi$, there is a set of $t\in[1,T]$ with  measure at least $T^{1-2/\log x}$, such that
$$\sum_{n\le x}n^{{\rm i}t}=x{\rm e}^{{\rm i}\theta}\rho\bigg(\frac{\log x}{\log_2T}\bigg)\bigg(1+O\bigg(\frac{1}{\log x}+\frac{\log x(\log_3T)^2}{(\log_2T)^2}\bigg)\bigg).$$
\end{thm}

Now we will introduce the resonance method for large zeta sums. This method can date back to Voronin's work in 1988, and developed significantly by Soundararajan \cite{Sound08}. Firstly, when $x$ is not very large compared with $\exp((\log T)^{\frac12})$, we can use the so-called ``long resonance method" to detect large values of  zeta sums. This method for character sums is due to Munsch \cite{Munsch}, which improves previous work of Hough \cite{Hough}.

\begin{thm}\label{thm1.3}
Let $\log T\le x\le\exp((\log T)^{\frac12})$, then we have 
    $$\max_{t\in[1,T]}\Big|\sum_{n\le x}n^{{\rm i}t}\Big|\ge \Psi\bigg(x,\Big(\frac14+o(1)\Big)\frac{\log T\log_2T}{\max\{\log_2x-\log_3T,\log_3T\}}\bigg).$$
\end{thm}

When $\log x$ is a small power of $\log T$, we can write the lower bound in a more compact way.

\begin{cor}
Let $\log x=(\log T)^\sigma$ for a fixed $0<\sigma<1/2$. Then we have
$$\max_{t\in[1,T]}\Big|\sum_{n\le x}n^{{\rm i}t}\Big|\ge \Psi\bigg(x,\Big(\frac{1}{2\sigma}+o(1)\Big)\log T\bigg).$$
\end{cor}
When $x$ is even smaller (power of $\log T$), we can write the lower bound more precisely. 
\begin{cor}
Let $x=(\log T)^A$ for some  $A>1$. Then we have
$$\max_{t\in[1,T]}\Big|\sum_{n\le x}n^{{\rm i}t}\Big|\ge \Psi\bigg(x,\Big(\frac{1}{2}+o(1)\Big)\frac{\log T\log_2T}{\log_3T} \bigg).$$
\end{cor}

 When $x$ is very close to $\exp((\log T)^{\frac12})$, we would use the method of Hough \cite{Hough}, which is essentially Soundararajan's resonnance method in \cite{Sound08}. 
\begin{thm}\label{thm1.4}
   Let $x=\exp\left(\tau\sqrt{\log T\log_2T}\right)$, with $\tau=(\log_2T)^{o(1)}$. Let $A,\tau'\in\mbr$ such that
   \[\tau=\int_A^\infty \frac{\mme^{-x}}{x}\dif x,\qquad \tau'=\int_A^\infty\frac{\mme^{-x}}{x^2}\dif x.\]
Then we have 
   \[\max_{t\in[1,T]}\Big|\sum_{n\le x}n^{{\rm i}t}\Big|\ge \sqrt{x}\exp\left((1+o(1))A(\tau+\tau')\sqrt{\frac{\log T}{\log_2 T}}\right)\]
\end{thm}

When $x$ is much larger than $\exp((\log T)^{\frac12})$ but smaller than $\sqrt T$, we can combine the resonance method with GCD sums (also called G\'al-type sums). This kind of method was originated from Aistleitner \cite{Aistle16}, and subsequently developed by Bondarenko and Seip \cite{BS17,BS18}, and  La Bret\`eche and Tenenbaum \cite{DT}.
\begin{thm}\label{thm1.5}
    Let $\exp((\log T)^{\frac12+\varepsilon})<x\le   T^{\frac12}$, then we have 
    $$\max_{t\in[1,T]}\Big|\sum_{n\le x}n^{{\rm i}t}\Big|\ge \sqrt x\exp\bigg((\sqrt2+o(1))\sqrt{\frac{\log(T/x)\log_3(T/x)}{\log_2(T/x)}}\bigg).$$
\end{thm}
When $x$ is larger than $\sqrt T$, we may need to look for the relation between the quantities $\frac{1}{\sqrt x}|\sum_{n\le x}n^{\mmi t}|$ and $\sqrt{\frac{x}{T}}|\sum_{n\le T/x}n^{\mmi t}|$. In comparison, the Poisson summation formula for character sums suggests the ``symmetry'': $\frac{1}{\sqrt x}|\sum_{n\le x}\chi(n)|\approx\sqrt{\frac{x}{q}}|\sum_{n\le q/x}\chi(n)|$  (very roughly speaking).

In fact, our Theorems \ref{thm1.1}, \ref{thm1.2} and \ref{thm1.4} can be generalized to the case of $\sum_{n\le x}f(n)n^{{\rm i}t}$ in  shorter interval $[T,2T]$, where $f(n)$ is any completely multiplicative function satisfying $|f(n)|=1$ for any $n\in{\mathbb N}^*$. Xu and Yang have the same treatment in \cite{XY23} and get a similar result of our Theorem~\ref{thm1.4}.

This article is organized as follows. We will present some preliminary lemmas in \S\ref{sec2}. We will prove Theorems \ref{thm1.1}-\ref{thm1.5} separately in   \S \ref{sec3}-\ref{sec7}.

\section{Preliminary lemmas}\label{sec2}

In this section, we give some lemmas that we will use later.

Here and throughout this paper, we will put $X_p$ as a sequence of i.i.d random variables equidistributed on the unit circle for all prime $p$. If $n=\prod_{i}p_{i}^{a_i}$, let $X_n=\prod_{i}X_{p_{i}}^{a_i}$ be multiplicative random variables.

\begin{lem}\label{lem:moment}
Let $r(n)$ be any bounded arithmetic function and $\mathbb{E}(\cdot)$ be the expectation. Then
$$\frac{1}{T}\int_{0}^{T}\bigg|\sum_{n\leq x} n^{\mmi t}r(n)\bigg|^{2k}\dif t=\mathbb{E}\left(\bigg|\sum_{n\leq x}X_n r(n)\bigg|^{2k}\right)+O\bigg(\frac{x^{2k}}{T}\bigg).$$
\end{lem}
\begin{proof}
    Rearranging the sum, we get
$$\frac{1}{T}\int_{0}^{T}\bigg|\sum_{n\leq x} n^{\mmi t}r(n)\bigg|^{2k}\dif t=\frac{1}{T}\int_{0}^{T}\sum_{n_1,\dots,n_k,m_1\dots,m_k\leq x}\left(\frac{n}{m}\right)^{\mmi t}r(n)\overline{r(m)}\dif t$$
and note that 
\begin{equation}\label{eq:intait}
    \frac{1}{T-1}\int_1^Ta^{\mmi t}\dif t=\begin{cases}
  1 & \text{if }a=1 \\
  O\left(\frac{1}{T}\right)&\text{if }a\neq 1.
 \end{cases}
\end{equation}
\end{proof}
\begin{lem}\label{lem:cplsum}
    Let $f(n)$ be a completely multiplicative function with $|f(n)|=1$. Let $2\leq x\leq \exp((\log T)^{\frac{1}{2}})$ and $y=\frac{\log T}{\log x(\log_2 T)^{8}}$. Then for $t$ in at least a set of measure  $T^{1-\frac{1}{(\log_2 T)^2}}$ we have
    $$\sum_{n\le x\atop n\in\mathcal{S}(x)}n^{\mmi t}=\sum_{n\le x\atop n\in\mathcal{S}(y)}f(n)+O\bigg(\frac{\Psi(x,y)}{(\log_2 T)^2}.\bigg)$$
\end{lem}
\begin{proof}
For $k\leq \lfloor\frac{1}{3}\frac{\log T}{\log x}\rfloor$, one has
$$\frac{1}{T}\int_{0}^{T}\bigg|\sum_{n\le x\atop n\in\mathcal{S}(y)}\frac{\overline{f(n)}n^{\mmi t}+1}{2}\bigg|^{2k}\dif t=\mathbb{E}\bigg(\bigg|\sum_{n\le x\atop n\in\mathcal{S}(y)}\frac{X_n+1}{2}\bigg|^{2k}\bigg)+O(T^{-\frac{1}{3}}).$$
We pick only those $X_n$ for which $|\arg X_p|\leq\frac{\pi}{\log T}$ for all $p\leq y$. For these choice
$$\sum_{n\le x\atop n\in\mathcal{S}(y)}\frac{X_{n}+1}{2}=\Psi(x,y)+O\bigg(\sum_{n\le x\atop n\in\mathcal{S}(y)}\frac{\Omega(n)}{\log T}\bigg)=\Psi(x,y)\bigg(1+O\Big(\frac{\log x}{\log T}\Big)\bigg).$$
With the same argument as the proof of Lemma 7.1 in~\cite{GS01}, we get the desired result.
\end{proof}
\section{Proof of Theorem \ref{thm1.1}}\label{sec3}
By Lemma~\ref{lem:moment}, for $k=\lfloor\frac{1}{3}\frac{\log T}{\log x}\rfloor$ we have
$$\frac{1}{T}\int_{0}^{T}\bigg|\sum_{n\leq x} n^{\mmi t}-\Psi(x,y;t)\bigg|^{2k}\dif t\leq \mathbb{E}\left(\bigg|\sum_{n\leq x}X_n -\Psi(x,y;X_n)\bigg|^{2k}\right).$$ 
According to ~\cite[Theorem 6.1]{GS01}, if $y\geq C\log^2 x$ then there are a certain constant $c>0$ so that
\begin{align*}
 \frac{1}{T}&\int_{0}^{T}\bigg|\sum_{n\leq x} n^{\mmi t}-\Psi(x,y;t)\bigg|^{2k}\dif t\\
 &\leq c^{k} \Psi(x,y)^{2k}\left(\frac{k\log x\log y}{y}\right)^{k}\exp\left(O\left(\frac{k\log^2 x\log_2 x}{y}\right)\right).
\end{align*}
For $A>1$, we deduce that there are $t$ of measure at most $TA^{-2k}$ not satisfying
$$\bigg|\sum_{n\leq x} n^{\mmi t}-\Psi(x,y;t)\bigg|\ll A\Psi(x,y)\left(\frac{k\log^2 x\log y}{y}\right)^{\frac{1}{2}}\exp\left(O\left(\frac{\log^2 x\log_2 x}{y}\right)\right).$$
Take $y=\log x\log T(\log_2 T)^5$ and $A=2$. We get the first part of Theorem~\ref{thm1.1}.

For the second part, let $y=(\log T+\log^2 x)(\log_2 T)^4$ and $A=\exp\big(\frac{\log x}{(\log_2 T)^2}\big).$ Then there are at most $t$ of measure $T^{1-1/(\log_2 x)^2}$ such that
$$\bigg|\sum_{n\le x}n^{{\rm i}t}\bigg|\le\Psi(x,(\log T+(\log x)^2)(\log_2T)^5) .$$
\section{Proof of Theorem \ref{thm1.2}}\label{sec4}
Put $y=\frac{\log T}{\log x(\log_2 T)^{8}}$ and $y_1=\log T (\log \log T)^7$. With the same argument as the proof of~\cite[Theorem 3]{GS01}, we get that except for a set of measure at most $T^{1-\frac{1}{\log x}}$  
\begin{align*}
\sum_{n\leq x}n^{\mmi t}
&=\Psi(x,y_1;t)+ O\bigg(\frac{\Psi(x,y_1)}{(\log \log T)^2}
\bigg)\\
&=\Psi(x,y;t) +O(|\Psi(x,y_1)-\Psi(x,y)|)
+O\bigg(\frac{\Psi(x,\log T)}{(\log \log T)^2}\bigg)\\
&= \Psi(x,y;t)+O\bigg( \Psi(x,\log T) \frac{\log x (\log \log \log T)^2}{(\log \log T)^2}  \bigg).
\end{align*}
Take $f(n)=n^{\frac{\mmi\theta}{\log x}}$ in Lemma~\ref{lem:cplsum}. We see that with an exceptional set of measure at most $T^{1-\frac{1}{\log x}}$,
$$
\Psi(x,y;t) =\sum_{n\leq x\atop n\in\mathcal{S}(y)}n^{\frac{\mmi\theta}{\log x}} + O\bigg(\frac{\Psi(x,y)}{(\log \log T)^2}\bigg)=\mme^{\mmi\theta} \Psi(x,y) + O\bigg(\frac{\Psi(x,\log T)}{\log x}
\bigg).
$$
\section{Proof of Theorem \ref{thm1.3}}\label{sec5}
Let $\varepsilon$ be a very small positive number. For $$y:=\Big(\frac14-\varepsilon\Big)\frac{\log T\log_2T}{\max\{\log_2x-\log_3T,\log_3T\}},$$ let $a_k$ be completely multiplicative with $a_1=1$, $a_p=1-\frac{\log y}{\log x(\log_2T)^{1+\delta}}$ for $p\le y$ and $a_p=0$ for $p>y$. Here $\delta$ is a positive number  smaller than $\varepsilon$.
We have the following result for $a_k$, which follows directly from \cite[P. 35-36]{Munsch}.
\begin{lem}\label{lem5.1}Let $a_k$ and $y$ be defined above. We have
    $$\sum_{k\le x\atop k\in\mcs(y)} a_k\ge\Psi(x,(1+o(1))y).$$
\end{lem}
Define the resonators
$$R(t):=\prod_{p\le y}\bigg(1-\frac{a_p}{p^{{\rm i}t}}\bigg)^{-1}=\sum_{  k\in\mcs(y)}\frac{a_k}{k^{\mmi t}},$$
where $a_k$ is defined above. We have
$$\log|R(t)|\le \log R(0)=-\sum_{p\le y}\log(1-a_p)\le\Big(\frac12-\varepsilon\Big)\log T.$$
Let $\phi(t):={\rm e}^{-t^2}$. Define
$$M_1(R,T):=\int_{1\le|t|\le T}|R(T)|^2\phi\bigg(\frac{t\log T}{T}\bigg)\dif t,$$
and 
$$M_2(R,T):=\int_{1\le|t|\le T}S_t(x)|R(T)|^2\phi\bigg(\frac{t\log T}{T}\bigg)\dif t,$$
where $S_t(x)=\sum_{n\leq x}n^{\mmi t}$. 
Trivially we have
$$M_1(R,T)\le\int_{\mathbb R}|R(T)|^2\phi\bigg(\frac{t\log T}{T}\bigg)\dif t:=I_1(R,T).$$
Since 
$$\int_{|t|\le 1}S_t(x)|R(T)|^2\phi\bigg(\frac{t\log T}{T}\bigg)\dif t\ll T^{1-\varepsilon},$$
by the upper bound of $R(t)$,
and
$$\int_{|t|\ge T}S_t(x)|R(T)|^2\phi\bigg(\frac{t\log T}{T}\bigg)\dif t\ll 1,$$
by the rapid decreasing of $\phi(\cdot)$,
we have 

$$M_2(R,T):=\int_{\mathbb R}S_t(x)|R(T)|^2\phi\bigg(\frac{t\log T}{T}\bigg)\dif t+O(T^{1-\varepsilon}):=I_2(R,T)+O(T^{1-\varepsilon}).$$
Thus 
$$\max_{1\le |t|\le T}|S_t(x)|\ge\frac{|M_2(R,T)|}{M_1(R,T)}\ge\frac{|M_2(R,T)|}{I_1(R,T)}=\frac{I_2(R,T)+O(T^{1-\varepsilon})}{I_1(R,T)}.$$
For $I_1(R,T)$, we have
$$I_1(R,T)=\sum_{\ell,n\in\mcs(y)}a_\ell a_n\int_{\mathbb R}\bigg(\frac \ell n\bigg)^{-\mmi t}\phi\bigg(\frac{t\log T}{T}\bigg)\dif t=\frac {T}{\log T}\sum_{\ell,n\in\mcs(y)}a_{\ell}a_n\widehat\phi\bigg(\frac {T}{\log T}\log\frac{\ell}{n}\bigg).$$
For $I_2(R,T)$, we have 
\begin{align*}
I_2(R,T)&=\sum_{k\le x}\sum_{m,n\in\mcs(y)}a_ma_n\int_{\mathbb R}\bigg(\frac mn\bigg)^{-\mmi t}\phi\bigg(\frac{t\log T}{T}\bigg)\dif t\\
&=\frac {T}{\log T}\sum_{k\le x}\sum_{m,n\in\mcs(y)}a_ma_n\widehat\phi\bigg(\frac {T}{\log T}\log\frac{m}{kn}\bigg)\\
&\ge\frac {T}{\log T}\sum_{k\le x\atop k\in\mcs(y)}\sum_{m,n\in\mcs(y)\atop k|m}a_ma_n\widehat\phi\bigg(\frac {T}{\log T}\log\frac{m}{kn}\bigg)\\
&=\frac {T}{\log T}\sum_{k\le x\atop k\in\mcs(y)}\sum_{\ell,n\in\mcs(y)}a_{k\ell}a_n\widehat\phi\bigg(\frac {T}{\log T}\log\frac{\ell}{n}\bigg)\\
&=\sum_{k\le x\atop k\in\mcs(y)} a_k I_1(R,T).
\end{align*}
So we deduce that
$$\frac{I_2(R,T)}{I_1(R,T)}\ge\sum_{k\le x\atop k\in\mcs(y)} a_k\ge\Psi(x,(1+o(1))y),$$
by Lemma \ref{lem5.1}, which completes the proof.

\section{Proof of Theorem \ref{thm1.4}}\label{sec6}

\begin{proof}
Let $y=T/x$ and $a=\sqrt{\log x\log_2x}$, we define a completely multiplicative function $r(n)$ by $r(p)=\frac{a}{\sqrt{p}\log p}$ where $a^2\leq p\leq \mme^{(\log a)^2}$ is prime and $r(p)=0$ for other primes. We define the resonator $R(t)=\frac{1}{\sqrt{T-1}}\sum_{n\leq y}r(n)n^{\mmi t}$, then we have 
\[\max_{t\in[1,T]}\Big|\sum_{n\le x}n^{{\rm i}t}\Big|\geq \frac{\left|\int_1^T|R(t)|^2S_t(x)\dif t\right|}{\int_1^T|R(t)|^2\dif t}.\]
By using the formula \eqref{eq:intait}, we have 
\begin{equation}\label{eq:rtdt}
\int_1^T|R(t)|^2\dif t=\sum_{n\leq y}r(n)^2+O\left(\frac{x^2}{T}\right),
\end{equation}
and 
\begin{equation}\label{eq:strtdt}
    \left|\int_1^T|R(t)|^2S_t(x)\dif t\right|=\sum_{\substack{ m,n\leq y\\ k\leq x\\ n=km}}r(m)r(n)+O\left(\frac{x^3}{T}\right)=\sum_{n\leq x\atop m\leq y/n}r(m)r(mn)+O\left(\frac{x^3}{T}\right).
\end{equation}
Now by combining Eq. \eqref{eq:rtdt} and Eq. \eqref{eq:strtdt} we get 
\begin{align*}
    \max_{t\in[1,T]}\left|\sum_{n\le x}n^{{\rm i}t}\right|&\geq \frac{\sum_{n\leq x}r(n)\sum_{m\leq y/n}r(m)^2}{\sum_{n\leq y}r(n)^2}+O\left(\frac{x^3}{T}\right)\\
    &\geq \frac{\sum_{m\leq y/x}r(m)^2}{\sum_{n\geq 1}r(n)^2}\sum_{n\leq x}r(n)+O\left(\frac{x^3}{T}\right).
\end{align*}
Finally, the result follows from the proof of Theorem $3.2$ in \cite[P.103]{Hough}

\end{proof}

\section{Proof of Theorem \ref{thm1.5}}\label{sec7}
Before we prove Theorem \ref{thm1.5}, we present the following result on GCD sums.
\begin{lem}\label{lem7.1}
    Let $\M$  be any set of positive integers satisfying $\max\M\le2\min\M$ and $|\M|=N$ be large. Then we have
    $$\max_{|\M|=N}\frac{1}{|\M|}\sum_{m,n\in\M}\sqrt{\frac{(m,n)}{[m,n]}}\ge \exp\bigg((2\sqrt2+o(1))\sqrt{\frac{\log N\log_3N}{\log_2N}}\bigg).$$
\end{lem}
\begin{proof}
This is a weaker version of \cite[Corollary 7.1]{DT}.
\end{proof}
Let $\M$  be a set of positive integers satisfying the conditions in Lemma \ref{lem7.1}, with cardinal $|\M|=N=\lfloor T/x\rfloor$. Define
$$
\M_j
:= \M\cap [(1+(\log T)/T)^j, (1+(\log T)/T)^{j+1}).
$$
For ${\mathcal J}:= \{j\ge0:\M_j\neq\varnothing\}$,
let $$\M'=\{m_j=\min\M_j:j\in{\mathcal J}\}.$$
Then we define the resonator
$$
R(t) := \sum_{m\in\M'} \frac{r(m)}{m^{\mmi t}},
$$
where $r(m_j)=|\M_j|^{1/2}.$ Trivially we have
$$
|R(t)|
\le R(0)
= \sum_{m\in\M'} r(m)
\le \Big(\sum_{m\in\M'}1\Big)^{1/2} \Big(\sum_{m\in\M'}r(m)^2\Big)^{1/2}
\le |\M'|^{1/2}|\M|^{1/2}\le N.
$$
We define
$$M_1(R,T):=\int_{1\le|t|\le T}|R(t)|^2\phi\bigg(\frac{t\log T}{T}\bigg)\dif t,$$
and 
$$M_2(R,T):=\int_{1\le|t|\le T}|S_t(x)R(t)|^2\phi\bigg(\frac{t\log T}{T}\bigg)\dif t.$$
For $M_1(R,T)$, we have 
$$M_1(R,T)\le I_1(R,T):=\int_{\mathbb R}|R(t)|^2\phi\bigg(\frac{t\log T}{T}\bigg)\dif t\ll\frac{T|\M|}{\log T}\le \frac{T^2}{x\log T}.$$
For $M_2(R,T)$, since 
$$\int_{|t|<1}|S_t(x)R(t)|^2\phi\bigg(\frac{t\log T}{T}\bigg)\dif t\ll x^2R(0)^2\le x^2N^2\le T^2,$$
and 
$$\int_{|t|>T}|S_t(x)R(t)|^2\phi\bigg(\frac{t\log T}{T}\bigg)\dif t\ll 1,$$
we have 
$$M_2(R,T)=I_2(R,T)+O(T^2):=\int_{\mathbb R}|S_t(x)R(t)|^2\phi\bigg(\frac{t\log T}{T}\bigg)\dif t+O(T^2).$$
Thus 
\begin{align}\max_{1\le |t|\le T}|S_t(x)|^2\ge\frac{|M_2(R,T)|}{M_1(R,T)}\ge\frac{I_2(R,T)}{I_1(R,T)}+O(x\log T)\gg\frac{\log T}{T|\M|}I_2(R,T).\label{reson7}\end{align}
Now we  focus on $I_2(R,T)$. We have
\begin{align}I_2(R,T)&=\sum_{k,l\le x}\sum_{m,n\in\M'}r(m)r(n)\int_{\mathbb R}\bigg(\frac {mk}{nl}\bigg)^{-\mmi t}\phi\bigg(\frac{t\log T}{T}\bigg)\dif t\nonumber\\
&=\frac{T}{\log T}\sum_{k,l\le x}\sum_{m,n\in\M'}r(m)r(n)\widehat\phi\bigg(\frac{T}{\log T}\log\frac{mk}{nl}\bigg).\label{I222}
\end{align}
For the inner sum, we have 
\begin{align*}
    \sum_{m,n\in\M'}r(m)r(n)\widehat\phi\bigg(\frac{T}{\log T}\log\frac{mk}{nl}\bigg)=&\sum_{i,j\in{\mathcal J}}r(m_i)r(m_j)\widehat\phi\bigg(\frac{T}{\log T}\log\frac{m_ik}{m_jl}\bigg)\\
    \geq &\sum_{i,j\in{\mathcal J}}\min\{r(m_i)^2,r(m_j)^2\}\widehat\phi\bigg(\frac{T}{\log T}\log\frac{m_ik}{m_jl}\bigg)\\
    =&\sum_{i,j\in{\mathcal J}}\widehat\phi\bigg(\frac{T}{\log T}\log\frac{m_ik}{m_jl}\bigg)\min\{|\M_i|,|\M_j|\}\\
    \ge&\sum_{i,j\in{\mathcal J}}\widehat\phi\bigg(\frac{T}{\log T}\log\frac{m_ik}{m_jl}\bigg)\sum_{m\in\M_i,n\in\M_j\atop mk=nl}1\\
    =&\sum_{i,j\in{\mathcal J}}\sum_{m\in\M_i,n\in\M_j\atop mk=nl}\widehat\phi\bigg(\frac{T}{\log T}\log\frac{m_in}{m_jm}\bigg).
\end{align*}
Since $m\in\M_i,n\in\M_j$ implies $\log\frac{m_in}{m_jm}\ll\frac{\log T}{T}$ and thus $\phi\bigg(\frac{T}{\log T}\log\frac{m_in}{m_jm}\bigg)\gg1,$ the above is 
$$\gg\sum_{i,j\in{\mathcal J}}\sum_{m\in\M_i,n\in\M_j\atop mk=nl}1=\sum_{m,n\in\M\atop mk=nl}1.$$
Inserting into \eqref{I222}, we have 
$$I_2(R,T)\gg\frac{T}{\log T}\sum_{k,l\le x}\sum_{m,n\in\M\atop mk=nl}1=\frac{T}{\log T}\sum_{m,n\in\M}\sum_{k,l\le x\atop mk=nl}1.$$
For fixed $m,n$, $mk=nl$ implies $k=nL/(m,n)$ and $l=mL/(m,n)$ for some integer $L$. Since $\max\M\le2\min\M$, we have for the inner sum
$$\sum_{k,l\le x\atop mk=nl}1\ge\frac{x}{\max\{\frac{m}{(m,n)},\frac{n}{(m,n)}\}}\ge\frac{x}{\sqrt{2\frac{m}{(m,n)}\frac{n}{(m,n)}}}=\frac{x}{\sqrt2}\sqrt{\frac{(m,n)}{[m,n]}}.$$
It follows that
$$I_2(R,T)\gg\frac{xT}{\log T}\sum_{m,n\in\M}\sqrt{\frac{(m,n)}{[m,n]}}.$$
By \eqref{reson7}, we have 
\begin{align*}\max_{1\le |t|\le T}|S_t(x)|^2&\gg\frac{\log T}{T|\M|}\frac{xT}{\log T}\sum_{m,n\in\M}\sqrt{\frac{(m,n)}{[m,n]}}\\&=x\frac{1}{|\M|}\sum_{m,n\in\M}\sqrt{\frac{(m,n)}{[m,n]}}\\
&\ge x\exp\bigg((2\sqrt2+o(1))\sqrt{\frac{\log(T/x)\log_3(T/x)}{\log_2(T/x)}}\bigg),
\end{align*}
where the last inequality follows from Lemma \ref{lem7.1}. Thus we complete the proof of Theorem \ref{thm1.5}.


\section*{Acknowledgements}
The authors would like to thank Bin Chen for drawing their attention to \cite{Harper}, and Yongxiao Lin for pointing to the paper \cite{La2019}. We also thank Daodao Yang for valuable comments on
the previous version of this article. The research of the third author was supported by Fundamental Research Funds for the Central Universities (Grant No. 531118010622).

\normalem

\end{document}